\documentclass[12pt]{amsart}
\usepackage{amssymb, amscd}
\usepackage[arrow,matrix,curve]{xy}
\usepackage{fullpage,enumerate}

\newcommand{\beq}{\begin{equation}}
\newcommand{\eeq}{\end{equation}}
\newcommand\dirac{\slash\!\!\!\partial}

\newcommand\R{\mathbb R}

\newcommand\Z{\mathbb Z}

\newcommand{\NN}{{\mathbb N}}
\newcommand{\QQ}{{\mathbb Q}}
\newcommand{\Q}{{\mathbb Q}}

\newcommand{\RR}{{\mathbb R}}
\newcommand{\TT}{{\mathbb T}}
\newcommand{\ZZ}{{\mathbb Z}}

\newcommand\Range{\operatorname{Range}}

\newtheorem{theorem}{Theorem}

\newtheorem{conjecture}{Conjecture}

\theoremstyle{definition}
\newtheorem{definition}{Definition}

\theoremstyle{remark}
\newtheorem{remark}{Remark}

\usepackage{color}
\definecolor{darkgreen}{cmyk}{1,0,1,.2}
\definecolor{m}{rgb}{1,0.1,1}
\definecolor{green}{cmyk}{1,0,1,0}

\definecolor{test}{rgb}{1,0,0}
\definecolor{cmyk}{cmyk}{0,1,1,0}

{

\marginparwidth 0pt
\oddsidemargin  0pt
\evensidemargin  0pt
\marginparsep 0pt
\topmargin   0pt
\textwidth 6.5 in \textheight 8.5 in

\begin{document}

\title[Proof of the magnetic gap-labelling conjecture for principal solenoidal tori]
{Proof of the magnetic gap-labelling conjecture for principal solenoidal tori}
\footnote{\today}
\author{Moulay Tahar Benameur}
\address{IMAG, Univ Montpellier, CNRS, Montpellier, France}
\email{moulay.benameur@umontpellier.fr}

\author{Varghese Mathai}
\address{Department of Pure Mathematics, University of Adelaide,
Adelaide 5005, Australia}
\email{mathai.varghese@adelaide.edu.au}

\begin{abstract}
In this note, we prove the magnetic spectral gap-labelling conjecture as stated in \cite{BM2018}, in all dimensions, for 
principal solenoidal tori. 
\end{abstract}



\maketitle

\tableofcontents


\section{Introduction}
In a recent paper \cite{BM2018}, the authors formulated for the first time the magnetic gap-labelling conjecture (MGL). More precisely, given a constant magnetic field on Euclidean space $\RR^p$ determined by a skew-symmetric $(p\times p)$ matrix $\Theta$,  and a $\Z^p$-invariant probability measure $\mu$ on the disorder set $\Sigma$ which is by hypothesis a Cantor  set, where the action 
of $\Z^p$ is assumed to be minimal, the corresponding Integrated Density of States of any self-adjoint operator
affiliated to the twisted crossed product  algebra $C(\Sigma) \rtimes_\sigma \ZZ^p$, where $\sigma$ is the multiplier on 
$\ZZ^p$ associated to  $\Theta$, takes on values on spectral gaps in the {\em magnetic gap-labelling group} cf. Definition \ref{defn:magnetic-groups}. The {\em magnetic frequency group} cf. Definition \ref{defn:magnetic-groups} is defined as an explicit  
countable subgroup of $\RR$ involving Pfaffians of $\Theta$ and its sub-matrices. Our the magnetic gap-labelling conjecture states that the magnetic gap labelling group is a subgroup of the magnetic frequency group.
In \cite{BM2018}, we gave evidence for the validity of our conjecture in 2D, 3D, the Jordan block diagonal case and the periodic case in all dimensions. The goal of this note is to give a complete proof of   the magnetic gap-labelling conjecture  in all dimensions, for 
principal solenoidal tori, see Theorem \ref{thm:integral} in section \ref{sec:integrality}. 
This result is obtained as a byproduct of a stronger result relating the magnetic gap-labelling conjecture with the integrality of the Chern character. 

\medskip

Historically speaking, the gap-labelling theorem was originally conjectured by Bellissard \cite{Bellissard4} in the late 1980s. It concerns the labelling of gaps in the spectrum of a 
Schr\"odinger operator (in the absence of a magnetic field) by the elements of a subgroup of $\RR$ 
which results from pairing the $K_0$-group of the noncommutative analog for the Brillouin zone with the tracial state defined by the probability measure on the hull.
The problem arises in a mathematical version of solid state physics in the context of aperiodic tilings. 
Its three proofs, discovered independently by the authors of \cite{BO,KP,BBG} all concern the proof of a statement in K-theory. See also \cite{BO-JFA}.
Earlier results include the proof of the gap-labelling conjecture in 1D \cite{Bellissard5}, 2D \cite{BCL,VanEst} and in 3D \cite{BKL}.
We remark that Bellissard's conjecture requires that the  gap labelling group is equal to the frequency group, as opposed to the subgroup relationship
in the case when the magnetic field is non-zero.

\medskip

The main results of the paper are in sections \ref{sect:proof of MGL} and \ref{sec:integrality}. In section \ref{sect:proof of MGL}, 
a proof of MGL is given assuming the integrality of the Chern character, an independent  result, which  extends the results of \cite{BO-JFA} and which was expected by many experts in the area. In section \ref{sec:integrality}, the magnetic gap-labelling conjecture for principal solenoidal tori is proved in full generality as an application of the previous result, by showing that the 
Chern character is indeed integral in this principal case. The MGL conjecture remains open for general solenoidal tori, see comments after the statement of Theorem \ref{thm:integral}.

\medskip

\noindent{\it Acknowledgements}. The authors thank Dennis Sullivan for useful discussions and for providing an interesting topological proof of the integrality of the Chern character for tori, 
which is not included in the paper.
MB thanks the French National Research Agency for support via the ANR-14-CE25-0012-01 (SINGSTAR), and  VM thanks the Australian Research Council  for support
via the ARC Laureate Fellowship FL170100020.

\section{Conjectures about the magnetic gap-labelling group}\label{sec:MGLconj}

We review in this first section the magnetic gap-labelling conjecture as stated in \cite{BM2018}, that we shall call MGL.

\subsection{{Some notations}}\label{Notations}
Assume that $\Z^p$ acts minimally on a Cantor set $\Sigma$ and let $\mu$ be a $\Z^p$-invariant probability measure on $\Sigma$.
The subgroup of $\R$ which is generated by $\mu$-measures of clopen subspaces of $\Sigma$ is denoted $\Z[\mu]$. This is known as the group of frequencies of  the aperiodic potential associated with the quasi-crystal, i.e.  appearing in the Fourier expansion of that potential. It can also be seen as the image under (the integral associated with) the probability measure $\mu$ of $C(\Sigma, \Z)$, the group of continuous integer valued functions on $\Sigma$. That is,
$$
\Z[\mu] = \left\{\int_\Sigma f(z) d\mu(z) \Big| f\in C(\Sigma, \Z)   \right\} = \mu(C(\Sigma, \ZZ))
$$ 
Recall that if the group $\Z^p$ acts on a module $M$, then the {\em coinvariants} of $M$ is defined as the quotient 
$$
{{M_{\Z^p}}} := M/\{m-gm|m\in M, g\in\Z^p\},
$$ 
and the {\em invariants}  of $M$ is defined as 
$$
{{M^{\Z^p}}}:=\{m\in M | m=gm\,  \text{for all}\,\, g\in \Z^p\}.
$$

Let $I$ be an ordered subset of $\{1,\ldots,p\}$ with an even number of elements, 
and let, {{according to the previous notations}},  $C(\Sigma, \ZZ)_{\ZZ^{I^c}}$ denote the coinvariants of $C(\Sigma, \ZZ)$ under the action of the subgroup $\ZZ^{I^c}$
of $\ZZ^p$, where $I^c$ denotes the index set that is the complement to $I$. 
{{Let {{similarly}} $\left(C(\Sigma, \ZZ)_{\ZZ^{I^c}}\right)^{\ZZ^I}$ 
denote the subset of $C(\Sigma, \ZZ)_{\ZZ^{I^c}}$ composed of those $\ZZ^{I^c}$-coinvariant classes in $C(\Sigma, \ZZ)_{\ZZ^{I^c}}$ which are invariant under the induced action of the subgroup $\Z^I$, {{or equivalently invariant under the action of the whole group $\Z^p$}}. }}
Define 
$$
\ZZ_I[\mu] = \mu\left(\left(C(\Sigma, \ZZ)_{\ZZ^{I^c}}\right)^{\ZZ^I}\right).
$$
{{Notice that 
$$
\ZZ_{\{1,\cdots, p\}}[\mu]={\ZZ \;  \subset  \; \ZZ_I[\mu] \; \subset  \; \ZZ[\mu]} = \ZZ_{\emptyset}[\mu].
$$}}

\subsection{Labelling the gaps}
Let $\sigma$ be a multiplier of $\Z^p$ which is associated with the skew symmetric matrix $\Theta$. The $\Z^p$ invariant probability measure $\mu$ yields a regular trace $\tau^\mu$ on the twisted crossed product $C^*$-algebra $C(\Sigma)\rtimes_\sigma \Z^p$, which is by fiat the operator norm completion of the $*$-algebra of compactly supported continuous functions
$C_c (\Z^p\times \Sigma)$ acting via the left regular representation on the Hilbert space $L^2(\Sigma, d\mu) \otimes \ell^2(\ZZ^d)$.
The trace $\tau^\mu$ is defined  on the dense subalgebra $C_c (\Z^p\times \Sigma)$ by the equality 
\beq\label{trace}
\tau^\mu (f)=\left\langle\mu, f_0\right\rangle  = \int_\Sigma f_0(z) d\mu(z), \quad \text{  where  } f_0: z\mapsto f(0, z), 
\eeq
with $0$ the zero element of $\Z^p$. Hence $\tau^\mu$ induces a trace
$$
\tau^\mu (= \tau^\mu_*)\; : \; K_0(C(\Sigma)\rtimes_\sigma \Z^p ) \longrightarrow \R.
$$
The $C^*$-algebra $C(\Sigma)\rtimes_\sigma \Z^p$ (or rather a Morita equivalent one) is a receptacle for the spectral projections onto spectral gaps of the magnetic Schr\"{o}dinger operator associated with our system, as explained earlier. Any gap in its spectrum  may therefore be labelled by the trace of the corresponding projection.

\begin{definition}\label{defn:magnetic-groups}
{{The}}  {\em{magnetic gap-labelling group}} is defined to be the range of the trace $\tau^\mu$
$$
 \Range (\tau^\mu) \; := \; \tau^\mu\left(K_0(C(\Sigma)\rtimes_\sigma \Z^p )\right) \; \subset \; \R.
$$ 
  \\

The {\em magnetic frequency group} is defined to be the countable subgroup of $\R$ given by:
$$
\sum_{0\leq |I|\leq p, |I even} {\rm Pf}(\Theta_I)\Z_I[\mu],
$$
where ${\rm Pf}(\Theta_I)$ denotes the Pfaffian of the skew symmetric $(|I|\times |I|)$ submatrix $\Theta_I$ of $\Theta$, whose $(i,j)$ entry is $\Theta_{(i,j)}$  for all $i, j \in I$.
\end{definition}

We next recall the formulation of our magnetic gap labelling conjecture from \cite{BM2018}, and later give evidence for its validity.

\begin{conjecture}[{\bf Magnetic gap-labelling conjecture (MGL)}]\label{mainconj}
Let $\Sigma$ be a Cantor set with a minimal action of $\ZZ^p$ that preserves a Borel probability measure $\mu$.
Let $\sigma$ be the multiplier on $\ZZ^p$ associated to a skew-symmetric $(p\times p)$ matrix $\Theta$. \\
Then the magnetic gap-labelling group is contained in the magnetic frequency group.\\
\end{conjecture}

{{When the magnetic field is trivial, the magnetic gap-labelling group coincides with the now classical gap-labelling group, while the magnetic frequency group reduces to the group $\ZZ[\mu]=\ZZ_\emptyset [\mu]$ recalled in Subsection \ref{Notations}. In this case, our conjecture does reduce to the Bellissard conjecture \cite{Bellissard4} which predicts the equality of these two groups. Indeed,  $\ZZ[\mu]$ is then obviously contained in the gap-labelling group, see \cite{BO, BO-JFA}.
}}

\section{{{Proof of MGL, assuming integrality of the Chern character}}}\label{sect:proof of MGL}

{\bf The integrality hypothesis}:
{\em \text{(IH)} \;  The range of the Chern character  $Ch: K^p(X) \longrightarrow H^{[p]}(X,\Q)\simeq \oplus_{k\geq 0} H^{p+2k} (\Z^p, C(\Sigma, \Q))$\text{ is contained in } $H^{p+2k} (\Z^p, C(\Sigma, \Z)).$}\\

We will show in the last section that (IH) is satisfied for principal solenoidal tori in all dimensions.

\begin{theorem}\label{MGL2}
Suppose that the integralty hypothesis {\rm{(IH)}} stated above is satisfied, then Conjecture  \ref{mainconj}, is true. 
\end{theorem}


\begin{proof}

By using the twisted index theorem (cf. \cite{BM2018}), we are reduced to proving that for a given multi-index $I:=(i_1, \cdots, i_{2k})$, the range of the map
$$
\cup \psi_I:= \cup (\psi_{i_1}\cup \cdots\cup \psi_{i_{2k}}) : H^{p-2k} (\Z^p, C(\Sigma, \Z))\longrightarrow H^{p} (\Z^p, C(\Sigma, \Z))\simeq C(\Sigma, \Z)_{\Z^p},
$$
lies inside $\left(C(\Sigma, \Z)_{\Z^{I^c}}\right)^{I}$. 

Our method is an induction on $p$. We denote by $(T_i)_{1\leq i\leq p}$ the generators of the action of $\Z^p$ on $\Sigma$. We have the following exact sequences (cf. \cite{Brown}
\begin{multline*}
0\to H^{p-2k-1} (<T_{i}, i\neq i_2>, C(\Sigma, \Z))_{<T_{i_2}>} \longrightarrow H^{p-2k} ( \Z^p, C(\Sigma, \Z)) \longrightarrow \\ H^{p-2k} (<T_{i}, i\neq i_2>, C(\Sigma, \Z))^{<T_{i_2}>} \to 0
\end{multline*}
and 
\begin{multline*}
0\to  H^{p-2k-1} (<T_{i}, i\neq i_1, i_2>, C(\Sigma, \Z))_{<T_{i_1}>} \longrightarrow H^{p-2k} (<T_{i}, i\neq i_2>, C(\Sigma, \Z))  \longrightarrow \\ H^{p-2k} (<T_{i}, i\neq i_1, i_2>, C(\Sigma, \Z))^{<T_{i_1}>}\to 0.
\end{multline*}
It is then clear that the map $\cup \psi_I$ vanishes on the range of $H^{p-2k-1} (<T_{i_j}, j\neq 2>, C(\Sigma, \Z))_{<T_{i_2}>}$ in the first exact sequence and yields a well defined morphism
$$
\alpha_I: H^{p-2k} (<T_{i}, i\neq i_2>, C(\Sigma, \Z))^{<T_{i_2}>} \longrightarrow C(\Sigma, \Z)_{\Z^p}.
$$
Moreover, this latter map is the restriction of $\cup (\psi_{i_2}\cup\cdots \cup \psi_{i_{2k}})$ to the invariants  $H^{p-2k} (<T_{i}, i\neq i_2>, C(\Sigma, \Z))^{<T_{i_2}>}$. Now a careful inspection of the second exact sequence allows to deduce that $\cup (\psi_{i_2}\cup\cdots \cup \psi_{i_{2k}})$ vanishes on the range of $H^{p-2k-1} (<T_{i}, i\neq i_1, i_2>, C(\Sigma, \Z))_{<T_{i_1}>}$ in this second exact sequence. Therefore, again the map $\cup (\psi_{i_2}\cup\cdots \cup \psi_{i_{2k}})$ induces a well defined morphism from $\left(H^{p-2k} (<T_{i}, i\neq i_1, i_2>, C(\Sigma, \Z))^{<T_{i_1}>}\right)^{<T_{i_1}>}$ to $C(\Sigma, \Z)_{<T_{i}, i\neq i_2>}$, which is easily seen to coincide with the restriction to 
$$
\left(H^{p-2k} (<T_{i}, i\neq i_1, i_2>, C(\Sigma, \Z))^{<T_{i_1}>}\right)^{<T_{i_1}>}
$$ 
of the morphism
$$
\cup(\psi_{i_3}\cup  \cdots\cup \psi_{i_{2k}}) : H^{p-2k} (<T_{i}, i\neq i_1, i_2>, C(\Sigma, \Z)) \longrightarrow H^{p-2} (<T_{i}, i\neq i_1, i_2>, C(\Sigma, \Z)).
$$
Notice that $H^{p-2} (<T_{i}, i\neq i_1, i_2>, C(\Sigma, \Z))\simeq C(\Sigma, \Z)_{<T_{i}, i\neq i_1, i_2>}$. Applying the induction hypothesis, we can deduce that the range of this latter factorizes through  
$$
\left(C(\Sigma, \Z)_{<T_{i}, i\neq i_1,\cdots, i_{2k}>}\right)^{<T_{i_j}, j=3,\cdots, 2k >}.
$$
So the restriction to the invariants under $T_{i_1}$ and $T_{i_2}$, allows to deduce that the range of the original map $\cup \psi_I$ is contained in 
$$
\left( C(\Sigma, \Z)_{<T_{i}, i\neq i_1,\cdots, i_{2k}>}\right)^{<T_{i_j}, j=1,\cdots, 2k >},
$$
hence the conclusion. 
\end{proof}

\section{Integrality of the Chern character: the case of the torus}

{{The integrality of the Chern character for the torus is standard. To be self-contained, and in the spirit of our approach to  the gap-labelling problem, we give a proof based on the index theorem which could be known to the experts.  }}
For $n$ even, let $E\to \TT^n$ be a complex vector bundle over the torus, and $\dirac_E$ denote the 
Dirac operator twisted by a connection on $E$.  Then since the torus is flat, the Atiyah-Singer index theorem \cite{AS}
says that 
$$
index(\dirac_E) = \int_{\TT^n} Ch(E).
$$
We deduce that the top degree component of the Chern character $Ch_n(E)$ is integral.

Let $L_{ij}\to \TT^n$ denote the line bundle over the torus such that $c_1(L_{ij})= dx_i \wedge dx_j$. 
Then 
$$
index(\dirac_{E\otimes L_{ij}} ) = \int_{\TT^n} Ch(E) (1+ dx_i\wedge dx_j)
$$
since $(dx_i\wedge dx_j)^k = 0$ for $k>1$. Since this is true for all $ 1\le i, j \le n$,  we conclude that the $(n-2)$ degree component of the Chern character,  $Ch(E)_{n-2}$ is integral. By induction, we see that all the components of $Ch(E)$ are integral.

In the case when $n$ is odd, consider the path of self adjoint Dirac operators $\dirac_t, \, t\in [0,1]$, where $\dirac_t = (1-t)\dirac \otimes I_N
+ t U (\dirac \otimes I_N) U^*$, where $U: \TT^n \to U(N)$ is a smooth map that represents an element in $K^1(\TT^n)$. Then the spectral flow of the family 
$\dirac_t, \, t\in [0,1]$ is given by the Atiyah-Singer index theorem \cite{AS}, \cite{Getzler}
$$
SF(\dirac_t) = \int_{\TT^n} Ch(U) 
$$
Again we conclude that the top degree component of the Chern character $Ch_n(U)$ is integral.

Let $\dirac_{L_{ij} t}, \, t\in [0,1]$ denote the path of self adjoint Dirac operators $$(1-t)\dirac_{L_{ij}} \otimes I_N
+ t U (\dirac_{L_{ij}} \otimes I_N) U^*.$$ Then the spectral flow of the family 
$\dirac_{L_{ij} t}, \, t\in [0,1]$ is given by the Atiyah-Singer index theorem,
$$
SF(\dirac_{L_{ij} t}) = \int_{\TT^n} Ch(U) (1+ dx_i\wedge dx_j)
$$
Since this is true for all $ 1\le i, j \le n$,  we conclude that the $(n-2)$ degree component of the Chern character,  $Ch(U)_{n-2}$ is integral. By induction, we see that all the components of $Ch(U)$ are integral.

  \section{Integrality of the Chern character for principal solenoidal tori}\label{sec:integrality}
  
  For each $j\in\NN$, let $X_j = \TT^n$. 
Define the finite regular covering map $f_{j+1} :X_{j+1} \to X_{j}$ to be
a finite covering map such that degree of $f_{j+1} $ is greater than 1 for all $j$. For example, let 
$f_{j+1} (z_1,...,z_n)=(z_1^{p_1},...,z_n^{p_n})$,where
each $p_j \in \NN\setminus \{1\}$. Set $(X_\infty,f_\infty)$ to be the inverse limit,
$\varprojlim (X_j,f_j)$. Then $X_\infty$ is a solenoid manifold, and $X_\infty \subset \prod_{j\in\NN} X_j$, where the right hand side  is a compact space when given the
Tychonoff topology, therefore $X_\infty$ is also compact. 
 Let $G_j = \ZZ^n/\Gamma_j$ be the finite covering space group of the finite cover 
 $p_j : X_j \to X_1 = \TT^n$.
Then the inverse limit $G_\infty= \varprojlim G_j = \varprojlim \ZZ^n/\Gamma_j$ is a profinite
Abelian group that is homeomorphic to the Cantor set, cf. Lemma 5.1 in  \cite{McCord}.
Moreover $G_\infty \to X_\infty \to \TT^n$ is a principal fibre bundle, cf. Theorem 5.6 in \cite{McCord}.
We call such an $X$ a {\em principal solenoidal torus}.

Now K-theory is continuous under taking inverse limits in the category of compact Hausdorff spaces, which follows from Proposition 6.2.9 in \cite{WO}, see also \cite{Phillips}, 
 $$
 K^{\bullet}(X_\infty) \cong \varinjlim K^{\bullet}(X_j). 
 $$
Now the Chern character 
 $$
 Ch : K^{\bullet}(X_j)  \to H^{\bullet}(X_j,\ZZ) 
 $$
 maps to integral cohomology, as shown earlier, since $X_j$ is a torus. Therefore 
 $$ 
 \varinjlim Ch :   K^{\bullet}(X_\infty) \cong\varinjlim K^{\bullet}(X_j)  \to  \varinjlim H^{\bullet}(X_j,\ZZ). \cong  H^{\bullet}(X_\infty, \ZZ)
 $$ 
 by the continuity for \v{C}ech cohomology under taking inverse limits in the category of compact Hausdorff spaces, cf. \cite{Spanier}.
Now $H^{\bullet}(X_j,\ZZ)$ are torsionfree Abelian groups, therefore the direct limit $\varinjlim H^{\bullet}(X_j,\ZZ)$  is again a torsionfree Abelian group.
 So we have proved the following,
 
 \begin{theorem}[Integrality of the Chern character]\label{thm:integral}
 Let $X$ be a principal solenoidal manifold over $\TT^n$ as above. Then the Chern character,
 $$ 
 \varinjlim  Ch :   K^\bullet(X_\infty)  \to    H^{\bullet}(X_\infty,\QQ)
 $$ 
  is integral, that is, the range is contained in $H^{\bullet}(X_\infty,\ZZ)$. 
 \end{theorem}
 In particular, the integrality hypothesis (IH) of section \ref{sect:proof of MGL} is true in this case.\\

 Conversely, let $\Sigma$ be a Cantor set. Then $\Sigma$ can be made into a profinite abelian group $G$ (in many ways, cf. \cite{Verjovsky2014}).
 Let $\rho: \ZZ^n \to G$ be a homomorphism with dense range. 
 Then we can form the principal $G$ bundle, 
 $G\to X\to \TT^n$, where $X= \RR^n \times_\rho G$.

 \begin{remark}
 Consider the countable nested family of finite index normal subgroups of $\ZZ$ defined by $p^j\ZZ, \, j \in \NN$. Then the inverse limit 
 $\varprojlim \ZZ/p^j\ZZ = \ZZ_p$ is the profinite group of p-adic integers and hence is a Cantor set. Consider $f_j : S^1 \to S^1$ defined by $f_j(z) = z^{p^j}$, whose kernel is 
 isomorphic to $\ZZ/p^j\ZZ$. Let $X=  \varprojlim  (S^1, f_j)$. Then $X$ is a principal $ \ZZ_p$-bundle over $S^1$
 , cf. \cite{McCord}, and is also an explicit example of a principal solenoidal torus.
 \end{remark}
 
  \begin{remark}
 There is an interesting connection between the general case and {\em number theory}. The Cantor group is 
the profinite completion $\widehat \ZZ^n$ of the Abelian lattice $\ZZ^n$.
 In number theory, $\widehat \ZZ^n$  is identified with the
Galois group of infinite separable field extensions, cf. \cite{Waterhouse}.

Therefore understanding the outer automorphism group of the Cantor
group is equivalent to understanding the outer automorphism group of the Galois group 
$\widehat \ZZ^n$. 
 \end{remark}

 A general minimal action of $\ZZ^n$ on $\Sigma$ is however more general, being a homomorphism $\phi: \ZZ^n \to {\rm Homeo}(\Sigma) = {\rm Aut}(G)$
 with dense orbits on $\Sigma$. It gives rise to a general solenoidal torus, that isn't necessarily principal. A solenoidal torus is an example of a solenoidal manifold as
 defined in \cite{Sullivan2014,Verjovsky2014} where some interesting results are proved regarding these.
 It remains an open question whether the Chern character is integral in this general case.


\end{document}